\documentclass[oneside]{amsart}
\usepackage[utf8]{inputenc}

\usepackage{euscript}
\usepackage[T2A]{fontenc}
\usepackage[utf8]{inputenc}
\usepackage[normalem]{ulem}

\usepackage[english]{babel} 
\usepackage{amsmath}
\usepackage{mathtools}
\usepackage{amsfonts}
\usepackage{amssymb, amsthm}
\usepackage{dsfont}

\usepackage{blindtext} %This package generates automatic text
\usepackage{epigraph}

\usepackage{hyperref}

\usepackage{color}

\usepackage{amsmath}
\usepackage{subfig}
\usepackage{geometry}
\usepackage{faktor}

\usepackage{multicol} %колонки

\usepackage{graphicx}
\graphicspath{}
\DeclareGraphicsExtensions{.pdf,.png,.jpg} 

\newcommand{\mbbR}{\mathbb{R}}

\newcommand{\mbbN}{\mathbb{N}}
\newcommand{\mbbZ}{\mathbb{Z}}

\newcommand{\mbbH}{\mathbb{H}}

\newcommand{\mbbF}{\mathbb{F}}

\newcommand{\eps}{\varepsilon}
\newcommand{\del}{\delta}

\newcommand{\suml}{\sum \limits}

\newcommand{\p}{\prime}

\newcommand{\wpart}[1]{{\left[ #1 \right]}}

\DeclareMathOperator{\diam}{diam}

\renewcommand{\le}{\leqslant}
\renewcommand{\ge}{\geqslant}

\theoremstyle{plain}
\newtheorem{thm}{Theorem}%Theorem} %[section], чтобы нумеровать сначала в каждом разделе
\newtheorem*{thm*}{Theorem}%Theorem}

\newtheorem*{lm*}{Lemma}%Lemma}
\newtheorem{lm}[thm]{Lemma}%Lemma}
\newtheorem{cor}[thm]{Corollary}%Corollary}
\newtheorem{prop}[thm]{Proposition}

\newtheorem*{claim*}{Claim}

\theoremstyle{definition}
\newtheorem{defn}[thm]{Definition}%Definition}
\theoremstyle{remark}
\newcommand{\lr}[1]{{\left( #1\right)}}

\DeclareMathOperator{\Aut}{Aut}

\newcommand{\acts}[1]{\stackrel{{{{#1}}}}{{\curvearrowright}}}

\begin{document}
%\selectlanguage{English}

\title{Scaling entropy growth gap}
\author{Georgii Veprev}
\date{\today}
\thanks{The work is supported by an RSF grant (project 21-11-00152).
}
\address{University of Geneva, Geneva, Switzerland}
\email{georgii.veprev@gmail.com}

\maketitle

{\rightline{\textit{To the 90'th anniversary of Anatoly Vershik}}}

\begin{abstract}
   The scaling entropy of a p.m.p. action is a slow-entropy type invariant that characterizes the intermediate growth of entropy in a dynamical system. An amenable group~$G$ has a scaling entropy growth gap if the scaling entropy of any its free p.m.p. action admits a non-trivial lower bound. We prove that the group $S_\infty$ of all finite permutations has a scaling entropy growth gap as well as the Houghton group $\mathcal{H}_2$. By proving this, we show that there are finitely generated amenable groups with this property.
\end{abstract}

\section{Introduction}

In this paper, we study slow entropy-type invariants of actions of amenable groups. We focus on the invariant proposed by Vershik called \emph{scaling entropy}~\cite{V10a, V10b, V12} although we also formulate our main results in terms of \emph{slow entropy} by Katok-Thouvenot~\cite{KT} and \emph{measure theoretic complexity} by Ferenczi~\cite{F}. All these invariants characterize an intermediate growth of entropy of a system which turns out to be an intrinsic characteristic that is invariant under measure-theoretic isomorphism. We give all the necessary definitions of the scaling entropy invariant in Section~\ref{section_scaling_entropy}.

A question that naturally arises is what are the possible asymptotics that may appear as values of such invariants. In recent papers~\cite{AGTW, S, Vep2, Vep22, R}, several attempts to answer such questions led to non-trivial applications. For the case of a single transformation, the possible values of scaling entropy were completely described in~\cite{PZ15, Z15b}. For amenable groups, this question becomes much harder. For any amenable group, it is shown in~\cite{Vep22} that the scaling entropy of a generic p.m.p. action of such a group can not be bounded from above by a given non-trivial sequence.  If the group is moreover residually finite (or, more generally, if it has a free compact action), then the scaling entropy of a generic action can not be bounded from below by a non-trivial rate. It was shown, however, that there are amenable groups that fail to satisfy the last property. Moreover, any free p.m.p. action of such a group satisfies a given lower bound on its scaling entropy. We call this property \emph{scaling entropy growth gap}.

Examples of amenable groups with a scaling entropy growth gap provided in~\cite{Vep22} are not finitely generated. Namely, they are the matrix groups $SL(2, {\bar\mbbF}_p)$ over the algebraic closure of a finite field~$\mbbF_p, \ p > 2$. The existence of the gap in these groups follows from Helfgott-type estimates of growth in finite linear groups (see, \cite{H, PS}).

In this paper, we provide an example of a \emph{finitely generated} amenable group such that the scaling entropy of any its free p.m.p. action has scaling entropy which is bounded from below by a given rate. It turns out that such an example can be obtained from the series of Houghton groups. To prove this we first show that the infinite symmetric group has a scaling entropy growth gap. 
\begin{thm}\label{theorem_sym}
    The group $S_\infty$ of all finite permutations has a scaling entropy growth gap.
\end{thm}
Note that~$S_\infty$ has only two finite-dimensional unitary representations: the sign of a permutation and the trivial one. Both of these do not distinguish even permutations and, therefore,~$S_\infty$ can not possess a free compact action. Theorem~\ref{theorem_sym} gives a quantitative version of this statement: any free action of $S_\infty$ is not only non-compact but also its scaling entropy is bounded from below by some given rate which does not depend on the action itself.

In section~\ref{section_closure} we prove that the property of scaling entropy growth gap extends from a subgroup to the ambient group. This allows us to provide an example of a finitely generated amenable group with this property. 

\begin{prop}\label{proposition_subgroup}
    Let~$H$ be a subgroup of an amenable group~$G$ and assume that~$H$ has a scaling entropy growth gap. Then the group~$G$ has a scaling entropy growth gap.
\end{prop}

Houghton groups $\mathcal{H}_n$ were introduced in~\cite{Hou78}. A particular example from this series of groups, namely, the Houghton group $\mathcal{H}_2$ consists of all bijective maps from $\mbbZ$ to $\mbbZ$, generated by translations and finite permutations.  This group is amenable and is generated by the shift map and an elementary involution exchanging~$0$ and~$1$.
Since the Houghton group $\mathcal{H}_2$ by definition contains $S_\infty$ we obtain the following theorem as a corollary.
\begin{thm}\label{theorem_Houghton}
    The Houghton group $\mathcal{H}_2$ has a scaling entropy growth gap.
\end{thm}

\subsection{Scaling entropy}\label{section_scaling_entropy}
As we mentioned above the main object we deal with in this paper is \emph{the scaling entropy} of a p.m.p. action. This invariant was introduced by Vershik in~\cite{V10a, V10b, V12}. The main idea of Vershik is to introduce a measurable metric on a measure space and then consider the evolution of its geometrical numerical characteristics. Namely, we consider an average of the metric and the minimal number of balls (with respect to this average) of radius~$\eps$ needed to cover the entire space up to measure~$\eps$. We call the logarithm of this number  \emph{the $\eps$-entropy} of a metric triple. The asymptotic behavior of this function, as was conjectured by Vershik and proved by Zatitskii in~\cite{Z15a}, does not depend on the choice of the initial metric and is a measure-theoretic invariant of a system. Similar ideas were introduced by Katok-Thouvenot~\cite{KT} and Ferenczi~\cite{F} and led to related slow-entropy type invariants. We refer the reader to a recent survey~\cite{VVZ} for details on the theory of scaling entropy and a survey~\cite{KKW} for several other slow-entropy type invariants. 

Now let us proceed with formal definitions. Let $(X,\mu)$ be a standard probability space. We call measurable function $\rho\colon (X^2,\mu^2) \to [0, +\infty)$ \emph{a measurable semimetric\footnote{The reader could also find the term "quasimetric" in the literature.}} if it is symmetric and satisfies the triangle inequality. We call~$\rho$ summable if $\rho \in L^1(X^2,\mu^2)$. For a positive number~$\eps$, we define \emph{the $\eps$-entropy} as follows. Let~$k$ be the minimal positive integer such that the space~$X$ can be decomposed into a union of measurable subsets~$X_0, X_1, \ldots, X_k$ with $\mu(X_0) < \eps$ and $\diam_\rho(X_i) < \eps$ for $i = 1, \ldots, k$. Then we define
\begin{equation}
    \mbbH_\eps(X,\mu, \rho) = \log_2 k.
\end{equation}
If there is no such~$k$ we say that the $\eps$-entropy of the triple is infinite. 
A semimetric is called \emph{admissible} if all its $\eps$-entropies are finite. It is shown in~\cite{VPZ} that a semimetric is admissible if and only if it is separable on a subset of full measure. For any measurable partition~$\xi$, there is the corresponding \emph{cut semimetric} $\rho_\xi$ which equals~$0$ if two points lie in the same cell of~$\xi$ and equals~$1$ otherwise. A cut semimetric is admissible if and only if the corresponding partition is essentially countable.

Let $G$ be an amenable group and $\lambda = \{F_n\}$ be a F\o lner sequence. Let $\alpha \colon G \to \Aut(X,\mu)$ be a~p.m.p. action of~$G$ on a standard propbability space $(X,\mu)$.  For a measurable semimetric~$\rho$, we define its average over~$F_n$:
\begin{equation}
  G_{av}^n\rho(x,y) = \frac{1}{|F_n|} \sum\limits_{g \in F_n} \rho(gx, gy) \quad x,y \in X.  
\end{equation}
We will also occasionally use the notation $G_{av}^{F_n}\rho$ to emphasize the set $F_n$ over which the average is computed. For an element $g \in G$ we denote the $g$-translation of~$\rho$  by $g^{-1} \rho$ meaning that $g^{-1} \rho (x, y) = \rho(gx, gy), \ x, y \in X$. A semimetric is called \emph{generating} if all its translations together separate points mod 0 that is for a subset~$X^\p$ of full measure and any distinct $x, y \in X^\p$ there exists some $g \in G$ such $g^{-1} \rho (x,y) > 0$.

Then we consider the following function of two variables
\begin{equation}
   \Phi_\rho(n,\eps) = \mbbH_\eps(X, \mu, G_{av}^n \rho).
\end{equation}
The main idea behind the scaling entropy invariant is that the asymptotic behavior in~$n$ of this function does not depend on the choice of the initial semimetric~$\rho$. To justify this we need to define what \emph{the asymptotic behavior} of~$\Phi_\rho$ is. We say that two functions $\Phi, \Psi \colon \mbbN \times \mbbR_+ \to \mbbR_+$ that decrease in their second arguments are \emph{asymptotically equivalent} if for any~$\eps$ there exists a $\delta>0$ such that 
\begin{equation}
 \Phi(n, \eps) \precsim \Psi(n,\delta) \text{ and } \Psi(n, \eps) \precsim \Phi(n,\delta),  
\end{equation}
when $n$ goes to infinity (here and in what follows, for two sequences~$f$ and~$g$, relation $f \precsim g$ means that $\limsup_n f(n) / g(n) < \infty$). In this case, we write $\Phi \asymp \Psi$. We denote by $\wpart{\Phi}$ the equivalence class of function~$\Phi$ with respect to relation~$\asymp$ and call it \emph{the asymptotic class}. We say that $\Psi$ \emph{asymptotically dominates}~$\Phi$ and write $\Phi \precsim \Psi$ if for any~$\eps$ there exists a $\delta>0$ such that $\Phi(n, \eps) \precsim \Psi(n,\delta)$ when~$n$ goes to infinity. Clearly, if $\Phi \precsim \Psi$ and $\Psi \precsim \Phi$ hold simultaneously then~$\Phi \asymp \Psi$. Relation~$\precsim$ agrees with the equivalence relation~$\asymp$ and can be naturally applied to its equivalence classes.  

\begin{thm}[\cite{Z15a, Z15b}]\label{theorem_invariance}
    Let $\rho$ and $\omega$ be two summable admissible generating semimetrics. Then the functions~$\Phi_\rho$ and $\Phi_\omega$ are asymptotically equivalent. 
\end{thm}

Theorem~\ref{theorem_invariance} gives rise to the following definition. 
\begin{defn}
    \emph{The scaling entropy} of a p.m.p. action of an amenable group~$G$ is the asymptotic class $\mathcal{H}(X,\mu, G, \lambda)$ of the function~$\Phi_\rho$ for some (hence for any) summable admissible generating semimetric~$\rho$.
\end{defn}
We will also use a shorter notation~$\mathcal{H}(\alpha, \lambda)$ for the scaling entropy of action~$\alpha$. Note that by definition $\mathcal{H}(\alpha, \lambda)$ depends on the choice of the F\o lner sequence~$\lambda$. 

For any action~$\alpha$ of an amenable group~$G$ there are two simple natural bounds for the scaling entropy (see, e.\,g.,~\cite{Vep2}):
\begin{equation}
  {1} \precsim \mathcal{H}(\alpha, \lambda) \precsim {|F_n|}.
  \end{equation}
Moreover, the scaling entropy is bounded (meaning that $\mathcal{H}(\alpha, \lambda) \asymp 1$) if and only if the action is compact (see~\cite{VPZ, YZZ}). Equivalence $\mathcal{H}(\alpha, \lambda) \asymp |F_n|$ characterizes exactly the systems with positive measure-theoretic entropy (see~\cite{Vep2}). 

For the classical case of a single transformation with $\lambda$ being the sequence of standard intervals $\{0, \ldots, n\}$, it is shown in~\cite{PZ15, Vep1} that the scaling entropy~$\mathcal{H}$ always contains a function decreasing in~$\eps$ and increasing and subadditive in~$n$. Moreover, for any such function~$\Phi$, there exists an ergodic transformation whose scaling entropy is exactly $\wpart{\Phi}$. Therefore, the complete description of possible values of the invariant in the classical case was given. 

For a general amenable group, the exact description of the set of possible asymptotic classes that may appear as scaling entropy for some p.m.p. action is unknown. However, recently several steps were made towards understanding the properties of this set (see~\cite{Lott, Vep22}). Some of these results led to several non-trivial applications to \emph{universal systems} and \emph{generic extensions}. As shown in~\cite{Vep22}, for any amenable group, F\o lner sequence $\lambda = \{F_n\}$, and any sequence $\phi(n) = o(|F_n|)$ there exists a free p.m.p. action~$\alpha$ with $\mathcal{H}(\alpha,\lambda) \not \precsim \phi$ (see also \cite{Lott} for a related result regarding slow entropy). This result can be interpreted as the absence of a gap between the maximal possible growth~$|F_n|$ and any other slower rate~$\phi$. For a residually finite group, the same holds for the lower bound: for any unbounded sequence $\phi(n)$ there exists a free p.m.p. action~$\alpha$ with $\mathcal{H}(\alpha,\lambda) \not \succsim \phi$. The last statement does not hold for a general amenable group, examples of such groups were given in~\cite{Vep22}. This shows that there is a significant difference in the behavior of our invariant for different amenable groups. This paper is devoted to studying this property closely and finding more examples of such groups, in particular, finitely generated examples.  

\subsection{Scaling entropy growth gap}
 
\begin{defn}[\cite{Vep22}]\label{definition_gap}
    We say that an amenable group $G$ has \emph{a scaling entropy growth gap} if for some (then for any) F\o lner sequence $\lambda = \{F_n\}$ there exists an unbounded function~$\phi(n)$ such that for any essentially free p.m.p. action $\alpha$ we have $\mathcal{H}(\alpha, \lambda) \succsim \phi$.
\end{defn}
The correctness of this definition, that is independence of the F\o lner sequence, is proved in~\cite{Vep22}. Here we can give an equivalent definition of the gap property replacing essentially free actions with essentially faithful and restricting our considerations to ergodic actions.

\begin{prop}\label{proposition_definition}
    Let $G$ be an amenable group and $\lambda = \{F_n\}_n$ its F\o lner sequence. Assume that there is an unbounded function $\phi(n)$ such that for every ergodic essentially free p.m.p. action~$\alpha$ of~$G$ the lower bound $\mathcal{H}(\alpha, \lambda) \succsim \phi(n)$ holds true. Then for every essentially faithful p.m.p. action~$\beta$ the same lower bound $\mathcal{H}(\beta, \lambda) \succsim \phi(n)$ holds true.
\end{prop}

\begin{proof}
    Let us denote by $(X, \mu)$ the measure space on which $\beta$ acts. Consider the infinite power~$\beta^\mbbN$ of the given action. For any non-identity element $g \in G$ we have $\mu(\{ x \colon \beta(g) x \not = x\}) > 0$ since~$\beta$ is essentially faithful. Then, clearly, $\mu^\mbbN(\{y \in X^\mbbN \colon \beta^\mbbN(g) y \not = y\}) = 1$, hence, the action $\beta^\mbbN$ is essentially free. It is easy to see that the scaling entropy of a countable power of a given p.m.p. action is the same, that is $\mathcal{H}(\beta, \lambda) = \mathcal{H}(\beta^\mbbN, \lambda)$.
    
    So, we may assume that $\beta$ is essentially free. The scaling entropy of a system dominates the scaling entropies of its ergodic components (see~\cite{VVZ}, the proof for amenable groups is the same) which are essentially free as well.  Since almost each ergodic component~$\mu_x$ satisfies  $\mathcal{H}(X, \mu_x, G, \lambda) \succsim \phi(n)$, we conclude that $\mathcal{H}(\beta, \lambda) \succsim \phi(n)$.
\end{proof}

As we already mentioned, residually finite amenable groups do not have a scaling entropy growth gap since they possess essentially free compact actions which have bounded scaling entropy. An example of an amenable group with a scaling entropy growth gap is the liner group  $SL(2, {\bar\mbbF}_p), \ p > 2$, as was shown in~\cite{Vep22} using Helfgott's estimates (see~\cite{H, PS}) in finite linear groups. In this paper, we provide more examples of such groups, in particular, a finitely generated one. Namely, we show that the infinite symmetric group $S_{\infty}$ of all finite permutations of integer numbers, and the Houghton group $\mathcal{H}_2$ have a scaling entropy growth gap. It is unknown to the author if there are amenable groups without an essentially free compact action that also have no scaling entropy growth gap.

\subsection{Two technical lemmas}
We will need several technical lemmas throughout our arguments. The following two estimates were proved in~\cite{PZ15}. 
\begin{lm}\label{lm_lowerbound}
        Let $\rho_1, \rho_2, \ldots, \rho_k$ be admissible semimetrics on~$\lr{X,\mu}$ such that almost surely $\rho_i \le 1$ for all $i = 1, \ldots, k$. Let $\tilde\rho = \frac{1}{k} (\rho_1+ \ldots+ \rho_k)$ and $\eps \in (0,1)$. Then there exists some $m \le k$ such that
            \begin{equation}
            \mbbH_{2\sqrt{\eps}}\lr{X,\mu, \rho_m} \le 
            \mbbH_{\eps}\lr{X,\mu, \tilde \rho}.
            \end{equation}
\end{lm}

\begin{lm}\label{lm_upperbound}
    Let $\rho_1, \rho_2, \ldots, \rho_k$ be admissible semimetrics on $(X, \mu)$ such that almost surely $\rho_i \le 1$ for all $i = 1, \ldots, k$. Assume that $\eps \in (0,1)$ satisfies $\mbbH_\eps(X, \mu, \rho_i) > 0$. Then the following inequality holds
    \begin{equation}
    \mbbH_{2\sqrt{\eps}}\left(X, \mu, 
     	\frac{1}{k} \suml_{ i= 1}^{k}\rho_i \right) \le
     	2 \suml_{ i= 1}^{k} \mbbH_\eps(X, \mu, \rho_i). 
    \end{equation}
\end{lm}

\section{Proof of theorem~\ref{theorem_sym}}
Although the group $SL(2, \overline{\mbbF}_q)$ can not be embedded into $S_\infty$ (in this case we would conclude due to Proposition~\ref{proposition_subgroup}), one can consider embeddings of its finite subgroups. This is the main idea of the proof:  to identify a linear group $SL(2, \mbbF_q)$, where $q = p^k$, with a subgroup of the symmetric group $S_{m_q}$, where $m_q = |SL(2, \mbbF_q)|$ via the permutation action of $SL(2, \mbbF_q)$ on itself. 

Let $g_0 \in SL(2, \mbbF_q)$ be the element represented by the matrix $\big(\begin{smallmatrix}
  1 & 1\\
  0 & 1
\end{smallmatrix}\big)$. Clearly, $g_0$ has order~$p$ and, therefore, is a disjoint union of $p$-cycles as a permutation of the elements of $SL(2, \mbbF_q)$. Then for every permutation $\sigma \in S_{m_q}$ which is a disjoint union of $p$-cycles there exists an embedding $ \tau \colon SL(2, \mbbF_q) \hookrightarrow S_{m_q}$ that maps $g_0$ to $\sigma$.
Further, let us define 
\begin{equation}
    \sigma^0 = (1, 2, \ldots, p) \in S_{m_q} \le S_\infty    
\end{equation}
be the cyclic permutation of the first~$p$ natural numbers, and
\begin{equation}
    \sigma_1 = (1, 2, \ldots, p)^2 (p+1, p+2, \ldots, 2p) \ldots (m_q - p + 1, m_q - p + 2, \ldots, m_q) \in  S_{m_q};   
\end{equation}
\begin{equation}
    \sigma_2 = (p, p -1, \ldots, 1) (2p, 2p-1, \ldots, p + 1) \ldots (m_q, m_q -1 \ldots, m_q - p + 1) = \sigma^0 \sigma_1^{-1}.   
\end{equation}
By definition, all three permutations $\sigma_0$, $\sigma_1$ and $\sigma_2$ commute and have order~$p$. Moreover, both $\sigma_1$ and $\sigma_2$ are (as elements of $S_{m_q}$) disjoint unions of $p$-cycles and, hence, we have two embeddings $\tau_1$ and $\tau_2$ of $SL(2, \mbbF_q)$ to $S_{m_q}$ that map $g_0$ to $\sigma_1$ and $\sigma_2$ respectively.

Now suppose that $S_\infty$ acts essentially freely on a standard measure space $(X,\mu)$. We can find a measurable partition~$\xi$ of~$X$ into~$p$ parts of equal measure, such that for almost every $x \in X$ the points~$x$ and~$\sigma^0 x$ lie in different cells of the partition~$\xi$. Let $\rho$ be a cut semimetric corresponding to~$\xi$. 

Assuming that $\mbbH_{\eps^2}(X,\mu, (S_{\infty})_{av}^{m_q}\rho) < \log k$ we can find an orbit $S_{m_q} x$ such that 
\begin{equation}
    \mbbH_{\eps}(S_{m_q} x,\nu, (S_{\infty})_{av}^{m_q}\rho) < \log k,    
\end{equation}
where $\nu$ is the uniform measure on~$S_{m_q} x$.
Note that the average $(S_{\infty})_{av}^{m_q}\rho$ restricted to $S_{m_q} x$ is the average over $S_{m_q}$ of $\rho$ restricted to~$S_{m_q} x$. Identifying the orbit~$S_{m_q} x$ with the symmetric group~$S_{m_q}$ we obtain a semimetric~$\rho$ on~$S_{m_q}$ with the following properties. First, for any $\sigma \in S_{m_q}$ the distance~$\rho(\sigma^0 \sigma, \sigma)$ is equal to~$1$. Second, the average 
\begin{equation}
   \Tilde{\rho} = \frac{1}{|S_{m_q}|} \sum_{\sigma \in S_{m_q}} \sigma^{-1} \rho 
\end{equation}
has $\eps$-entropy less than~$\log k$. 

 For any $\sigma \in S_{m_q}$ we have 
 \begin{equation}
    \rho(\sigma_1^{-1} \sigma, \sigma_2 \sigma) = \rho(\sigma_1^{-1} \sigma, \sigma_0 \sigma_1^{-1} \sigma) = 1.
\end{equation}
 Hence, either $\rho(\sigma_1^{-1} \sigma,  \sigma) = 1$ or $\rho(\sigma_2 \sigma,  \sigma) = 1$. One of these equalities holds for at least half of the elements $\sigma$ of $S_{m_q}$. Without loss of generality assume that there exists a subset $A \subset S_{m_q}$, such that $|A| \ge \frac{1}{2} |S_{m_q}|$, and  $\rho(\sigma_2 \sigma,  \sigma) = 1$. Then, there exists a subset $B \subset A$ such that $|B| \ge \frac{1}{2p} |S_{m_q}|$ and $\rho(x,y) = 0$ for any $x, y \in B$, since~$\rho$ is a cut semimetric corresponding to a partition into $p$ parts. Consider a family~$\mathcal{C}$ of $\tau_2(SL(2, \mbbF_q))$ cosets that intersect~$B$ by at least $\frac{1}{4p} m_q$:
 \begin{equation}
     \mathcal{C} = \big\{ \{\tau_2(SL(2, \mbbF_q)) \sigma\} \colon  |\tau_2(SL(2, \mbbF_q)) \sigma \cap B| \ge \frac{1}{4p} m_q\big\}.
 \end{equation}
 Clearly, the cardinality of $\mathcal{C}$ is at least $\frac{1}{4p} \frac{|S_{m_q}|}{m_q}$. Indeed, otherwise, the number of elements in~$B$ would not exceed
 \begin{equation}
   |\mathcal{C}| \cdot m_q + \frac{1}{4p} m_q \cdot \frac{|S_{m_q}|}{m_q} < \frac{1}{2p} |S_{m_q}|.
\end{equation}
 Since $\sigma_2 \in \tau_2(SL(2, \mbbF_q))$, we have for each $C \in \mathcal{C}$  the inclusion  $\sigma_2 (B\cap C) \subset C$ and, moreover, $\rho(x, y) = 1$ for any two points $x \in (B\cap C)$ and $ y \in  \sigma_2 (B\cap C)$.
 Therefore the mean value of the semimetric~$\rho$ restricted to~$C$ is at least $(\frac{1}{4p})^2$. Hence, restriction to~$C$ of the $\tau_2(SL(2, \mbbF_q))$-average
 \begin{equation}
    \Tilde{\rho}^q = \frac{1}{m_q} \sum_{\sigma \in \tau_2(SL(2, \mbbF_q))} \sigma^{-1} \rho
\end{equation}
 has mean value at least $(\frac{1}{4p})^2$ as well due to $\tau_2(SL(2, \mbbF_q))$-invariance of~$C$.
Since 
\begin{equation}
   \Tilde{\rho} = \frac{1}{|S_{m_q}|} \sum_{\sigma \in S_{m_q}} \sigma^{-1} \Tilde{\rho}^q, 
\end{equation}
due to Lemma~\ref{lm_lowerbound}, the $2\sqrt{\eps}$-entropy of $\Tilde{\rho}^q$ is bounded from above by the $\eps$-entropy of $\Tilde{\rho}$ that is $\log k$. 

Let $X_0, X_1, \ldots, X_k$ be the corresponding partition of $S_{m_q}$ realizing the $2\sqrt{\eps}$-entropy of $\Tilde{\rho}^q$. That is  $|X_0| < 2\sqrt{\eps} |S_{m_q}|$ and $\diam_{\Tilde{\rho}^q} X_i < 2\sqrt{\eps}$ for $i = 1, \ldots, k$.
There are at least $(1 - \sqrt[4]{\eps}) |\mathcal{C}|$ cosets from~$\mathcal{C}$ intersecting~$X_0$ by less than $8p\sqrt[4]{\eps} m_q$. For any such coset $C$, the $8p\sqrt[4]{\eps}$-entropy of $\Tilde{\rho}^q$ restricted to~$C$ is bounded from above by~$\log k$. Let us fix any such coset $C \in \mathcal{C}$.

Now, let us recall the following proposition from \cite{Vep22}.
\begin{prop}\label{prop_sl2}
Let $\rho$ be a left--invariant semimetric on $SL(2, {\mbbF_{q}})$ with diameter greater than~$3\del$, where $\del \in (0,\frac{1}{2})$. Then $\mbbH_\del(SL(2, {\mbbF_{q}}), \nu, \rho) \ge c\log q $, where~$\nu$~is the uniform probability measure and~$c$ is an absolute constant.
\end{prop}
 Since $\tau_2$ is an embedding we can identify the coset~$C$ with~$SL(2, {\mbbF_{q}})$. Taking $\eps < (\frac{1}{384 p^3})^4$ we obtain that the diameter of $\Tilde{\rho}^q$ restricted to~$C$ is at least its mean value which is at least $24p \sqrt[4]{\eps}$. Finally, its $8p\sqrt[4]{\eps}$-entropy does not exceed   $\log k$ and, hence, due to Proposition~\ref{prop_sl2} we have 
 \begin{equation}
    \log k \ge c \log q \ge c_0 \log \log |S_{m_q}|, 
 \end{equation}
  where $c_0$ is an absolute constant. Therefore, for the F\o lner sequence $\lambda = \{S_{m_{p^k}}\}_k$ and any essentially free p.m.p. action~$\alpha$ of~$S_\infty$ we have 
  \begin{equation}\label{eq_lowboundsym}
      \mathcal{H}(\alpha, \lambda) \succsim \log \log |S_{m_{p^k}}|,
  \end{equation}
    hence,  $S_\infty$ has a scaling entropy growth gap.
\section{Closure properties}\label{section_closure}

In this section, we study several closure properties of the family of groups that have a scaling entropy growth gap. We start with the proof of Proposition~\ref{proposition_subgroup} that gives the last remaining part of the proof of Theorem~\ref{theorem_Houghton}.

\begin{proof}[Proof of Proposition~\ref{proposition_subgroup}]
    Let $\{W_n\}_n$ be a F\o lner sequence in~$G$ and $\{F_n\}_n$ a F\o lner sequence in~$H$. Let~$\phi(n)$ be a corresponding absolute lower bound for scaling entropy of $(H, \{F_n\}_n)$. Let $\alpha$ be an essentially free p.m.p. action of~$G$ on a standard probability space~$(X, \mu)$ and let $\beta$ be its restriction to $H$.
    Let~$\rho$ be a measurable admissible metric on~$(X,\mu)$ bounded from above by $1$ almost everywhere. Since~$\alpha$ is essentially free, $\beta$ is essentially free as well. Therefore, due to Proposition~\ref{proposition_definition} we have for any $\del > 0$ sufficiently small 
    \begin{equation}\label{eq_98457694}
        \mbbH_\del(X, \mu, H_{av}^{F_n}) \succsim \phi(n).
    \end{equation}
    For any integer~$n$, there exits some~$k_n$ such that for any $r > k_n$ the inequality 
    \begin{equation}\label{eq_6745665}
        |F_n W_{r} \triangle W_{r}| < 2^{-n} | W_{r}|
    \end{equation}
    is satisfied. Then
    \begin{equation}
        \frac{1}{|W_{r}|}\suml_{g \in W_{r}} g^{-1} \frac{1}{|F_n|}\suml_{h \in F_n} h^{-1} \rho  
        \le \frac{1}{|W_r|} \suml_{f \in F_n W_r} f^{-1} \rho = G_{av}^{W_r} \rho + l_1,
    \end{equation}
    where the term~$l_1$ is bounded in absolute value by~$2^{-n}$. The last equality holds true due to inequality~\eqref{eq_6745665}. 
    Take~$\eps > 0$ satisfying 
    $\mbbH_{4\sqrt{\eps}}(H_{av}^{F_n}\rho) \succsim \phi(n)$. For sufficiently large~$n$, the term~$l_1$ is negligible for computing $\eps$--entropy of~$G_{av}^{W_{r}} \rho$. Lemma~\ref{lm_lowerbound} together with relation~\eqref{eq_98457694} give
    \begin{equation}
        \mbbH_\eps(G_{av}^{W_{r}}\rho) \ge 
        \mbbH_{4\eps}(G_{av}^{W_r} \rho + l_1)  \ge
        \mbbH_{4\sqrt{\eps}}(G_{av}^{F_n}\rho) \succsim \phi(n).
    \end{equation}
    Therefore, $G$~has a scaling entropy growth gap with respect to~$\{W_r\}$ and bound function $\psi(r) = \phi(n(r))$, where~$n(r)$ is the maximal~$n$ such that $k_n < r$.
    \end{proof}
    
\begin{prop}
    Let $H$ be a finite index subgroup of an amenable group $G$ and assume that $G$ has a scaling entropy growth gap. Then the subgroup $H$ has a scaling entropy growth gap.
\end{prop}

\begin{proof}
    Since $H$ has a finite index, the group $G$ decomposes as a finite disjoint union of right cosets. That is there are $g_1, \ldots, g_k \in G$ , where $k$ is the index of~$H$ in~$G$, such that  $G = \bigcup_i Hg_i$. 
    Let~$\{W_n\}_n$ be~a~F\o lner sequence in $G$ which is simultaneously left and right F\o lner. We have 
    \begin{equation}
        W_n = \bigcup\limits_{i= 1}^k (W_n \cap Hg_i).
    \end{equation}
    Find an index $i_0$ for which the set $W_n \cap Hg_i$ has the largest cardinality among these intersections. By passing to a subsequence, we can assume that~$i_0$ does not depend on~$n$. Let 
    \begin{equation}
    F_n = W_n g_{i_0}^{-1} \cap H \subset H.
    \end{equation}
    Clearly $|F_n| \ge \frac1k |W_n|$ and, hence, $\{F_n\}_n$ is a left F\o lner sequence in~$H$ since $W_n$ is a left F\o lner sequence in~$G$. Moreover, since $W_n$ is also a right F\o lner sequence we have for any~$\eps > 0$ and any sufficiently large~$n$ 
    \begin{equation}
        |F_n g_i \triangle (W_n \cap Hg_i)| < \eps k |F_n|,
    \end{equation}
    due to $g_{i_0}^{-1} g_i$ right almost invariance of $W_n$. Therefore, the set 
    \begin{equation}
        \Tilde{W}_n = \bigcup\limits_{i = 1}^k F_n g_i
    \end{equation}
    satisfies $|\Tilde{W}_n \triangle W_n | = o(|W_n|)$ and, hence, is a F\o lner sequence in~$G$.
    
    Now let $H \acts{\alpha} (X,\mu)$ be an essentially free action of~$H$. Let $G \acts{\beta} \prod_{i = 1}^k g_i^{-1}(X,\mu) = (Y, \nu)$ be the corresponding coinduced action of $\alpha$ from~$H$ to~$G$ (see, e.\,g., \cite{KL}). Let $\rho$ be a bounded admissible metric on $(X,\mu)$. Then the function $\Tilde{\rho}(x,y) = \rho(x_1, y_1)$ is an admissible generating semimetric on~$Y$. We have 
    \begin{equation}\label{eq_84567834}
      \frac{1}{|\tilde{W}_n|} \sum_{g \in \tilde{W}_n} g^{-1} \Tilde{\rho} = \frac{1}{k} \sum_{i = 1}^k g_i^{-1} \frac{1}{|F_n|}\sum_{h \in F_n} h^{-1} \Tilde{\rho}.   
    \end{equation}
    Clearly, for any $i = 1, \ldots, k$ and any $\eps > 0$
    \begin{equation}
        \mbbH_\eps (Y, \nu, \frac{1}{|F_n|}\sum_{h \in F_n} h^{-1} \Tilde{\rho}) = \mbbH_\eps(X, \mu, H_{av}^{F_n} \rho).
    \end{equation}
    And, since the right hand side of~\eqref{eq_84567834} is a finite average, we have due to Lemma~\ref{lm_upperbound}
    \begin{equation}
        \mathcal{H}(G, \beta, \{\Tilde{W}_n\}) \precsim \mathcal{H}(H, \alpha, \{F_n\}_n).
    \end{equation}
    Since by our assumptions $G$~has a scaling entropy growth gap, there exists a function $\phi(n)$ tending to infinity (which depends only on the F\o lner sequence $\{\Tilde{W}_n\}_n$) such that $\mathcal{H}(G, \beta, \{\Tilde{W}_n\}) \succsim \phi$. Then
    \begin{equation}
        \mathcal{H}(H, \alpha, \{F_n\}_n) \succsim \phi.
    \end{equation}
    Therefore, $H$ has a scaling entropy growth gap as well since $\alpha$ was chosen as an arbitrary essentially free action of~$H$. 
\end{proof}

\begin{cor}
    The infinite Alternating group $A_\infty$ of all even permutations of integer numbers has a scaling entropy growth gap.
\end{cor}

\begin{prop}
    Let $H$ and $G$ be amenable groups without a scaling entropy gap. Then the direct product $G \times H$ has no scaling entropy gap. 
\end{prop}

\begin{proof}
    Let $\lambda = \{W_n\}_n$ be a F\o lner sequence in~$G$ and $\sigma = \{F_n\}_n$ a F\o lner sequence in $H$. Then $\tau = \{W_n \times F_n\}_n$ is a F\o lner sequence in $G \times H$. Let $\phi(n)$ be a given function increasing to infinity. Find an essentially free p.m.p. action $G \acts{\alpha} (X,\mu)$ satisfying $\mathcal{H}(\alpha, \lambda) \not \succsim \phi$. 
    There exists a subsequence $\{n_k\}_k$ such that for any $\Phi \in \mathcal{H}(\alpha, \lambda)$ and any $\eps > 0$
    \begin{equation}
     \Phi(n_k, \eps) \prec \phi(n_k).
    \end{equation}
    Let $\lambda^\p = \{W_{n_k}\}_k$, $\sigma^\p = \{F_{n_k}\}_k$, $\tau^\p = \{W_{n_k} \times F_{n_k}\}_k$, and $\tilde\Phi(k, \eps) = \Phi(n_k, \eps)$.  Find an essentially free p.m.p. action $H \acts{\beta} (Y,\nu)$ satisfying $\mathcal{H}(\beta, \sigma^\p) \not \succsim \phi(n_k)$. Consider the product action~$\alpha \times \beta$ of~$G \times H$ on~$(X,\mu) \times (Y,\nu)$, where~$G$ and~$H$ act independently via actions~$\alpha$ and~$\beta$ on the first and the second components respectively. For any $\Psi \in \mathcal{H}(\beta, \sigma^\p)$ we have $\tilde{\Phi} + \Psi \in \mathcal{H}(\alpha \times \beta, \tau^\p)$. Since $\tilde{\Phi} \prec \phi(n_k)$ and $\Psi \not \succsim \phi(n_k)$ we conclude that
    \begin{equation}
        \mathcal{H}(\alpha \times \beta, \tau^\p) \not \succsim \phi(n_k),
    \end{equation}
    and, therefore,  $\mathcal{H}(\alpha \times \beta, \tau) \not \succsim \phi(n)$ since~$\tau^\p$ is a subsequence of~$\tau$.
\end{proof}

\section{Remarks and open questions}
\subsection{Connections to other slow entropy type invariants}
The scaling entropy growth gap property can be expressed in terms of related slow entropy-type invariants. Here we give a simple alternative definition of this property in terms of slow entropy introduced in~\cite{KT} by Katok and Thouvenot. We refer the reader to survey~\cite{KKW} for definitions of this and similar invariants.
\begin{prop}\label{proposition_eqiv_def}
    An amenable group $G$ has a scaling entropy growth gap if and only if for some (hence for any) F\o lner sequence in~$G$  there exists a non-trivial scale~$\mathbf{a}_n(t)$ such that for any essentially free p.m.p. action~$\alpha$ the lower slow entropy of $\alpha$ (with respect to the given F\o lner sequence) is positive:
    $
    \underline{ent}_{\mathbf{a}}(\alpha) > 0.
    $
\end{prop}

The proof is straightforward and follows directly from the definitions of scaling entropy and slow entropy. See a recent survey~\cite{VVZ} for more on the connections between these invariants. In view of Proposition~\ref{proposition_eqiv_def}, one may call our property a "slow entropy growth gap". In a similar manner, one can give an alternative definition of this property via the notion of measure-theoretic complexity introduced by~Ferenczi in~\cite{F}.

\subsection{Open questions}
We conclude the paper with several open questions and possible directions for further investigations of the scaling entropy growth gap property. 

First, there are lots of interesting amenable groups for which it is unknown to the author whether they have this property or not. The classical lamplighter group $\mbbZ_2 \wr \mbbZ$ is residually finite and hence has no scaling entropy growth gap, however, it is unknown if a lamplighter group $A \wr \mbbZ$ with non-abelian finite group $A$ of the lamps has this property. Other interesting candidates are finitely generated simple amenable groups, the first examples of which were constructed in~\cite{JM}.

Another question is whether the possession of essentially free compact action (residual finiteness for finitely generated groups) is the only obstruction to the scaling entropy growth gap property. For now, groups with such actions are the only examples of groups known to the author that do not have a scaling entropy growth gap. 

In a recent discussion with the author, Anatoly Vershik proposed the following question which is naturally related to the property of the scaling entropy growth gap. For a given group without compact actions, what are the (essentially free) actions with \emph{the slowest possible} scaling entropy? One may consider this question as trying to find an action of a group that is the closest in this sense to being compact.  In other words, can one identify the optimal bound~$\phi(n)$ in Definition~\ref{definition_gap} and the actions for which it is achieved? Since the absence of compact actions is determined by the family of finite-dimensional representations of a group, this question should be related to the representation theory of the group as well. For~$SL(2, \overline{\mbbF}_p)$ with a specific F\o lner sequence this question was answered in~\cite{Vep22}. Even for the infinite symmetric group, this question remains open.

\section*{Acknowledgements} The author is sincerely grateful to Anatoly~Vershik for his guidance, support, and attention to the author's work on this topic for many years, as well as for being an inexhaustible source of new research problems, and for giving inspiration to attack them.
The author is also grateful to Pavel~Zatitskii for his useful comments and help in preparing this manuscript.

\bibliographystyle{amsplain}

\begin{thebibliography}{100}

    \bibitem{AGTW} T.~Austin, E~Glasner, J.-P.~Thouvenot, B.~Weiss, \emph{An ergodic system is dominant exactly when it has positive entropy},   \emph{Ergodic Theory and Dynamical Systems}, 1--15, 2022.
	
	\bibitem{F} S. Ferenczi, \emph{Measure-theoretic complexity of ergodic systems}, Israel Journal of Mathematics 100, 187--207, 1997.	
		
	\bibitem{H} H. A. Helfgott, \emph{Growth and generation in SL2(Z/pZ)}, Ann. of Math. (2), 167(2):601–623, 2008.

        \bibitem{Hou78} C. H. Houghton,  \emph{The first cohomology of a group with permutation module coefficients}, Arch. Math 31, 254–258, 1978.
	
        \bibitem{JM}  K. Juschenko, N. Monod,  \emph{Cantor systems, piecewise translations and simple amenable groups}, Annals of Mathematics, 178, 775–787, 2013.
	
	\bibitem{KT} A. Katok, J.-P. Thouvenot, \emph{Slow entropy type invariants and smooth realization of commuting measure-preserving transformations}, Annales de Institut Henri Poincare 33, 323--338, 1997.
	
	\bibitem{KKW} A. Kanigowski, A. Katok, D. Wei, \emph{Survey on entropy-type invariants of sub-exponential growth in dynamical systems}, \url{https://arxiv.org/abs/2004.04655v1}
	
	\bibitem{KL} D. Kerr, H. Li, \emph{Ergodic Theory: Independence and Dichotomies}, Springer, 2017.

        \bibitem{Lott} A. Lott, \emph{Zero entropy actions of amenable groups are not dominant}. \emph{Ergodic Theory and Dynamical Systems}, 1-16, 2023.
        
	\bibitem{VPZ} F. V. Petrov, A. M. Vershik, P. B. Zatitskiy, \emph{Geometry and dynamics of admissible metrics in measure spaces}, Central European Journal of Mathematics, 11 (3), 379--400, 2013.
    
	
	\bibitem{PZ15} F. V. Petrov, P. B. Zatitskiy, \emph{On the subadditivity of a scaling entropy sequence}, Journal of Mathematical Sciences, 215:6, 734--737, 2016.
	
	 \bibitem{PS} L. Pyber, E. Szab\'o, \emph{Growth in finite simple groups of Lie type}. J. Amer. Math. Soc., 29(1):95–146, 2016.
    
    \bibitem{R} V.~V.~Ryzhikov, \emph{Compact families and typical entropy invariants of measure-preserving actions}, Trans. Moscow Math. Soc., 82, 117–123, 2021.
        
    \bibitem{S} J. Serafin, \emph{Non-existence of a universal zero-entropy system}, Israel Journal of Mathematics 194, no. 1, 349--358, 2013.
	    
    \bibitem{Vep1} G.~Veprev, \emph{Scaling Entropy of Unstable Systems}, Journal of Mathematical Sciences, 254:2, 109-119, 2021.
    
    \bibitem{Vep2} G.~Veprev, \emph{Non-existence of a universal zero entropy system for non-periodic amenable group actions}, Israel Journal of Mathematics, 253, 715–743, 2023.
    
    \bibitem{Vep3} G.~Veprev, \emph{The scaling entropy of a generic action}, Journal of Mathematical Sciences, 261, 595–600, 2022.

    \bibitem{Vep22} G.~Veprev. \emph{Non-existence of a universal zero entropy system via generic actions of almost complete growth}, \url{https://arxiv.org/abs/2209.01902}.

    \bibitem{VVZ} A. Vershik, G. Veprev, P. Zatitskii, \emph{Dynamics of metrics in measure spaces and scaling entropy}. Russian Mathematical Surveys, 78(3), 53-114, 2023.
    
    {\bibitem{V10a} A.~M.~Vershik, \emph{Information, entropy, dynamics, in: Mathematics of the 20th Century: A View from Petersburg} [in Russian], MCCME, pp. 47– 76,  2010.}
    \bibitem{V10b} A. M. Vershik, \emph{Dynamics of metrics in measure spaces and their asymptotic invariants}, Markov Processes and Related Fields, 16:1, 169--185, 2010.
    
    \bibitem{V12} A. M. Vershik, \emph{Scaling entropy and automorphisms with pure point spectrum}, St. Petersburg Math. J., 23:1, 75--91, 2012.
			    
    \bibitem{YZZ} T.~Yu, G.~Zhang, R.~Zhang, \emph{Discrete spectrum for amenable group actions}, Discrete \& Continuous Dynamical Systems, 41(12):5871, 2021.
    
    \bibitem{Z15a} P. B. Zatitskiy, \emph{Scaling entropy sequence: invariance and examples}, Journal of Mathematical Sciences, 209:6, 890--909, 2015.
	
	\bibitem{Z15b} P. B. Zatitskiy, \emph{On the possible growth rate of a scaling entropy sequence}, Journal of Mathematical Sciences, 215:6, 715--733, 2016.
\end{thebibliography}

\end{document}